\documentclass
{amsart}
\usepackage{graphicx}
\newtheorem{thm}{Theorem}[section]
\theoremstyle{definition}
\newtheorem{cor}[thm]{Corollary}
\newtheorem{prop}[thm]{Proposition}
\newtheorem{defn}[thm]{Definition}
\newtheorem{lem}[thm]{Lemma}

\newtheorem{ex}[thm]{Example}

\numberwithin{equation}{section}
\begin{document}

\title[Fully $S$-idempotent modules]{Fully $S$-idempotent modules}
\author{Faranak Farshadifar}
\address{Department of Mathematics, Farhangian University, Tehran, Iran.}
\email{f.farshadifar@cfu.ac.ir}

\begin{abstract}
Let $R$ be a commutative ring with identity and $S$ be a multiplicatively closed subset of $R$.
The aim of this paper is to introduce the notion of fully $S$-idempotent modules as a generalization of fully idempotent modules and investigate some properties of this class of modules.
\end{abstract}


\maketitle

\section{\bf Introduction}
\vskip 0.4 true cm
Throughout this paper $R$ will denote a commutative ring with
identity and $\Bbb Z$ will denote the ring of integers.

Let $M$ be an $R$-module. $M$ is said to be a \emph{multiplication module} \cite{Ba81} if for every submodule $N$ of $M$, there exists an ideal $I$ of $R$ such that $N=IM$. It is easy to see that $M$ is a multiplication module if and only if $N=(N:_RM)M$ for each submodule $N$ of $M$.
 A submodule $N$ of $M$ is said to be \emph{idempotent} if $N=(N:_RM)^2M$.  Also, $M$ is said to be \emph{fully idempotent} if every submodule of $M$ is idempotent \cite{AF122}.

Let $S$ be a multiplicatively closed subset of $R$. In \cite{ATUS2020}, the authors introduced and investigated the concept of $S$-multiplication modules as a generalization of multiplication modules.  An $R$-module $M$ is said to be an \emph{$S$-multiplication module} if for each submodule $N$ of $M$, there exist $s \in S$ and an ideal $I$ of $R$
such that $sN \subseteq IM \subseteq N$ \cite{ATUS2020}.

Let $S$ be a multiplicatively closed subset of $R$. In this paper, we introduce the notions of fully $S$-idempotent $R$-modules as a generalization of fully idempotent modules and provide some useful information concerning theis new class of modules.
We say that a submodule $N$ of an $R$-module $M$ is an \emph{$S$-idempotent submodule} if there exists an $s \in S$ such that $sN\subseteq  (N:_RM)^2M \subseteq N$ (Definition \ref{d2.1}).
We say that an $R$-module $M$ is a \emph{fully $S$-idempotent module} if every submodule of $M$ is an $S$-idempotent submodule (Definition \ref{d2.2}).
Clearly every fully idempotent $R$-module is a fully $S$-idempotent $R$-module (Proposition \ref{p2.3}). Example \ref{e2.4} shows that the converse is not true in general.
In Theorem \ref{t2.8}, we characterize the fully idempotent $R$-modules.
Also, we characterize the fully $S$-idempotent $R$-modules, where $S$ satisfying the maximal multiple condition (Theorem \ref{t2.11}).  Let $M_i$ be an $R_i$-module for $i=1,2,..,n$ and let $S_1, ..., S_n$ be multiplicatively closed subsets of
$R_1, ..., R_n$, respectively. Assume that $M = M_1\times ...\times M_n$, $R = R_1\times ...\times R_n$ and
$S = S_1\times ...\times S_n$. Then we show that the following statements are equivalent.
\begin{itemize}
\item [(a)] $M$ is a fully $S$-idempotent module.
\item [(b)] $M_i$ is a fully $S_i$-idempotent module for each $i\in \{1, 2, ..., n\}$.
\end{itemize}
Also, among other results, it is shown that (Theorem \ref{t2.18}) if
$M$ an $S$-multiplication $R$-module and $N$ be a submodule of $M$, then the following statements are equivalent.
\begin {itemize}
\item [(a)] $N$ is an $S$-pure submodule of $M$.
\item [(b)] $N$ is an $S$-multiplication $R$-module and
$N$ is an $S$-idempotent submodule of $M$.
\end {itemize}
Finally, we prove that
 if $M$ is a fully $S$-idempotent $R$-module,  then $M$ is a fully $S$-pure $R$-module. The converse holds if $M$ is an $S$-multiplication $R$-module (Corollary \ref{c2.20}).
\section{\bf Main results}
\begin{defn}\label{d12.1}
Let $S$ be a multiplicatively closed subset of $R$. We say that an element $x$ of an $R$-module $M$ is an \emph{$S$-idempotent element} if there exist $s \in S$ and $a \in (Rx:_RM)$ such that $sx=ax$.
\end{defn}

\begin{defn}\label{d2.1}
Let $S$ be a multiplicatively closed subset of $R$. We say that a submodule $N$ of an $R$-module $M$ is an \emph{$S$-idempotent submodule} if there exists an $s \in S$ such that $sN\subseteq  (N:_RM)^2M \subseteq N$.
\end{defn}

\begin{defn}\label{d2.2}
Let $S$ be a multiplicatively closed subset of $R$. We say that an $R$-module $M$ is a \emph{fully $S$-idempotent module} if every submodule of $M$ is an $S$-idempotent submodule.
\end{defn}

\begin{ex}\label{e2.6}
 Let $S$ be a multiplicatively closed subset of $R$ and $M$ be an $R$-module with $Ann_R(M) \cap S\not=\emptyset$. Then clearly, $M$ is a fully $S$-idempotent $R$-module.
\end{ex}

\begin{prop}\label{p2.3}
Let $S$ be a multiplicatively closed subset of $R$. Every fully idempotent $R$-module is a fully $S$-idempotent $R$-module. The converse is true if $S \subseteq U(R)$, where $U(R)$ is the set
of units in $R$.
\end{prop}
\begin{proof}
This is clear.
\end{proof}

The following example shows that the converse of Proposition \ref{p2.3} is not true in general.
\begin{ex}\label{e2.4}
Take the $\Bbb Z$-module $M=\Bbb Z_{p^\infty}$ for a prime number $p$. Then we know that all proper submodules of $M$ are of the form $G_t=\langle 1/p^t+\Bbb Z\rangle$ for some $t \in \Bbb N \cup \{0\}$ and $(G_t:_{\Bbb Z}M)=0$. Therefore, $M$ is not a fully idempotent $\Bbb Z$-module. Now,
take the multiplicatively closed subset $S = \{p^n: n \in \Bbb N \cup \{0\}\}$ of $\Bbb Z$. Then $p^tG_t=0 \subseteq (G_t:_{\Bbb Z}M)^2M\subseteq G_t$. Hence, $G_t$ is an $S$-idempotent submodule of $M$ for each  $t \in \Bbb N \cup \{0\}$. So, $M$ is a fully $S$-idempotent $\Bbb Z$-module.
\end{ex}

\begin{lem}\label{l2.5}
Let $S$ be a multiplicatively closed subset of $R$. Then every fully $S$-idempotent $R$-module is an $S$-multiplication $R$-module.
\end{lem}
\begin{proof}
Let $M$ be a fully $S$-idempotent $R$-module and $N$ be a submodule of $M$. Then there is an $s \in S$ such that
$sN\subseteq  (N:_RM)^2M \subseteq N$. This implies that
$$
sN\subseteq  (N:_RM)^2M= (N:_RM)(N:_RM)M \subseteq (N:_RM)M\subseteq N,
$$
as needed.
\end{proof}

The following example shows that the converse of Lemma \ref{l2.5}
is not true in general.

\begin{ex}\label{e2.7}
Take the multiplicatively closed subset $S=\Bbb Z \setminus 2\Bbb Z$ of $\Bbb Z$. Then
 $\Bbb Z_{4}$ is an $S$-multiplication
  $\Bbb Z$-module. But $\Bbb Z_{4}$ is not a fully $S$-idempotent $\Bbb Z$-module.
\end{ex}

Let $S$ be a multiplicatively closed subset of $R$. Recall that the saturation $S^*$ of $S$ is defined as $S^*=\{x \in R : x/1\  is \ a\ unit  \ of\ S^{-1}R \}$. It is obvious that $S^*$ is a multiplicatively closed subset of $R$ containing $S$ \cite{Gr92}.
\begin{prop}\label{p2.8}
Let $S$ be a multiplicatively closed subset of $R$ and $M$ be an R-module. Then we have the following.
\begin{itemize}
\item [(a)] If $S_1 \subseteq S_2$ are multiplicatively closed subsets of $R$ and $M$ is a fully $S_1$-idempotent $R$-module, then $M$ is a fully $S_2$-idempotent $R$-module.
\item [(b)] $M$ is a fully $S$-idempotent $R$-module if and only if $M$ is a fully $S^*$-idempotent $R$-module.
\item [(c)] If $M$ is a fully $S$-idempotent $R$-module, then every submodule of $M$ is a fully $S$-idempotent  $R$-module.
\end{itemize}
\end{prop}
\begin{proof}
(a) This is clear.

(b) Let $M$ be a fully $S$-idempotent $R$-module. Since $S\subseteq S^*$,  by part (a), $M$ is a fully $S^*$-idempotent $R$-module.
For the converse, assume that $M$ is a fully $S^*$-idempotent module and $N$ is a submodule of $M$.
Then there exists $x \in S^*$ such that $xN \subseteq (N :_R M)^2M$. As $x \in S^*$,  $x/1$ is a unit of $S^{-1}R$ and so $(x/1)(a/s)=1$ for some $a \in R$ and $s \in S$. This yields that $us = uxa$ for some $u \in S$. Thus we have $usN =uxaN \subseteq xN \subseteq (N :_R M)^2M$. Therefore, $M$ is a fully $S$-idempotent $R$-module.

(c) Let $N$ be a submodule of $M$ and $K$ be a submodule of $N$. By Lemma \ref{l2.5},
$M$ is an $S$-multiplication $R$-module. Hence there exists $s \in S$ such that $sK\subseteq (K:_RM)^2M \subseteq K$ implies
that
$$
s^2K\subseteq s(K:_RM)^2M \subseteq s(K:_RM)K\subseteq (K:_RM)(K:_RM)^2M \subseteq (K:_RM)^3M.
$$
Thus
$$
s^2K\subseteq (K:_RM)^3M \subseteq (K:_RN)^2(N:_RM)M \subseteq (K:_RN)^2N.
$$
Therefore, $N$ is fully $S$-idempotent.

\end{proof}

In the following theorem, we characterize the fully idempotent $R$-modules.
\begin{thm}\label{t2.8}
Let $M$ be an R-module. Then the following statements are equivalent:
\begin{itemize}
\item [(a)] $M$ is a fully idempotent $R$-module;
\item [(b)] $M$ is a fully $(R-\mathfrak{p})$-idempotent $R$-module for each prime ideal $\mathfrak{p}$ of $R$;
\item [(c)] $M$ is a fully $(R-\mathfrak{m})$-idempotent $R$-module for each maximal ideal $\mathfrak{m}$ of $R$;
\item [(d)] $M$ is a fully $(R-\mathfrak{m})$-idempotent $R$-module for each maximal ideal $\mathfrak{m}$ of $R$ with $M_{\mathfrak{m}}\not=0_{\mathfrak{m}}$.
\end{itemize}
\end{thm}
\begin{proof}
$(a) \Rightarrow (b)$
 Let $M$ be a fully idempotent $R$-module and $\mathfrak{p}$ be a prime ideal of $R$. Then $R-\mathfrak{p}$ is multiplicatively
closed set of $R$ and so $M$ is a fully $(R-\mathfrak{p})$-idempotent $R$-module by Proposition \ref{p2.3}.

$(b) \Rightarrow (c)$
Since every maximal ideal is a prime ideal, the result follows from the part (b).

$(c) \Rightarrow (d)$
This is clear.

$(d) \Rightarrow (a)$
Let $N$ be a submodule of $M$. Take a maximal ideal $\mathfrak{m}$ of $R$ with $M_{\mathfrak{m}}\not=0_{\mathfrak{m}}$. As
$M$ is a fully $(R-\mathfrak{m})$-idempotent module, there exists $s \not \in \mathfrak{m}$ such that $sN\subseteq  (N:_RM)^2M$. This implies that
$$
N_{\mathfrak{m}}=(sN)_{\mathfrak{m}}\subseteq  ((N:_RM)^2M)_{\mathfrak{m}}\subseteq N_{\mathfrak{m}}.
$$
If $M_{\mathfrak{m}}=0_{\mathfrak{m}}$, then clearly $N_{\mathfrak{m}}=((N:_RM)^2M)_{\mathfrak{m}}$. Thus we conclude
that $N_{\mathfrak{m}}=((N:_RM)^2M)_{\mathfrak{m}}$ for each maximal ideal $\mathfrak{m}$ of $R$. It follows that $N=(N:_RM)^2M$,  as needed.
\end{proof}

\begin{prop}\label{p2.9}
Let $S$ be a multiplicatively closed subset of $R$ and $N$ be an $S$-idempotent submodule of an $R$-module $M$. Then there is an $s \in S$ such that
$$
sN\subseteq Hom_R(M,N)N=\sum\{\varphi(N) : \varphi:M\rightarrow N\}.
$$
\end{prop}
\begin{proof}
As $N$ is an $S$-idempotent submodule of $M$, there is an $s \in S$ such that $sN\subseteq  (N:_RM)^2M$ and so $sN\subseteq  (N:_RM)N$. Let $x \in N$. Then there exists $r \in (N:_RM)$ and $y \in N$ such that $sx=ry$. Now consider the homomorphism $f:M\rightarrow N$ defined by $f(m)=rm$. Then we have
$$
sx=f(y) \in \sum\{\varphi(N) : \varphi:M\rightarrow N\}=Hom_R(M,N)N.
$$
Therefore, $sN\subseteq Hom_R(M,N)N$.
 \end{proof}

The following example shows that the converse of Proposition \ref{p2.9}
is not true in general.
\begin{ex}\label{e2.10}
Let $p$ be a prime number. Take the multiplicatively closed subset $S=\Bbb Z \setminus p\Bbb Z$ of $\Bbb Z$. Then one can see that the submodule $N=\Bbb  Z_p\oplus 0$ of the $\Bbb Z$-module $M=\Bbb
  Z_p\oplus \Bbb Z_p$ is not $S$-idempotent but $sN\subseteq Hom_{\Bbb Z}(M,N)N=N$ for each $s \in S$.
\end{ex}

A multiplicatively closed subset $S$ of $R$ is said to satisfy the \textit{maximal multiple condition} if there
exists an $s \in S$ such that $t\mid s$ for each $t \in S$.

In the following theorem, we characterize the fully $S$-idempotent $R$-modules, where $S$ is a multiplicatively closed subset of $R$ satisfying the maximal multiple condition.
\begin{thm}\label{t2.11}
Let $S$ be a multiplicatively closed subset of $R$ satisfying the maximal multiple condition
(e.g., $S$ is finite or $ S \subseteq U(R)$) and let $M$ be an $R$-module. Then the following statements are equivalent:
\begin{itemize}
  \item [(a)] $M$ is a fully $S$-idempotent module;
  \item [(b)] Every cyclic submodule of $M$ is $S$-idempotent;
  \item [(c)] Every element of $M$ is $S$-idempotent;
  \item [(d)] For all submodules $N$ and $K$ of $M$, we have
               $s(N \cap K)\subseteq (N:_RM)(K:_RM)M$ for some $s \in S$.
\end{itemize}
 \end{thm}
\begin{proof}
$(a) \Rightarrow (b)$  and $(b) \Rightarrow (c)$ are clear.

$(c) \Rightarrow (a)$. Let $N$ be a submodule of $M$ and $x \in N$. Then by
hypothesis, there exist $s_x \in S$ and $a \in (Rx:_RM)$ such that $s_xx=ax$. Hence $as_xx=a^2x$ and so
$s_x^2x=s_xax=as_xx=a^2x$. Thus $s_x^2Rx\subseteq (Rx:_RM)^2M$.
Now as  $S$ satisfying the maximal multiple condition, there
exists an $s \in S$ such that  $sRx\subseteq (Rx:_RM)^2M\subseteq (N:_RM)^2M$.
Therefore,  $sN \subseteq (N:_RM)^2M$, as required.

$(a) \Rightarrow (d)$. Let $N$ and $K$ be two submodules of $M$. Then for some $s \in S$ we have
$$
s(N \cap K)\subseteq (N \cap K:_RM)^2M \subseteq (N:_RM)(K:_RM)M.
$$

$(d) \Rightarrow (a)$. For a submodule $N$ of $M$,
we have
$$
sN=s(N \cap N)\subseteq (N:_RM)(N:_RM)M=(N:_RM)^2M
$$
for some $s \in S$.
\end{proof}

Let $R_i$ be a commutative ring with identity, $M_i$ be an $R_i$-module for each $i = 1, 2,..., n$, and $n \in \Bbb N$. Assume that
$M = M_1\times M_2\times ...\times M_n$ and $R = R_1\times R_2\times ...\times R_n$. Then $M$ is clearly
an $R$-module with componentwise addition and scalar multiplication. Also,
if $S_i$ is a multiplicatively closed subset of $R_i$ for each $i = 1, 2,...,n$,  then
$S = S_1\times S_2\times ...\times S_n$ is a multiplicatively closed subset of $R$. Furthermore,
each submodule $N$ of $M$ is of the form $N = N_1\times N_2\times...\times N_n$, where $N_i$ is a
submodule of $M_i$.
\begin{thm}\label{l2.12}
Let $M_i$ be an $R_i$-module and $S_i \subseteq R_i$ be a multiplicatively closed subset for $i = 1, 2$. Assume that
$M = M_1\times  M_2$, $R = R_1\times  R_2$, and $S = S_1\times  S_2$. Then $M$ is a fully $S$-idempotent $R$-module if and only
if $M_1$ is a fully $S_1$-idempotent $R_1$-module and $M_2$ is a fully $S_2$-idempotent $R_2$-module.
\end{thm}
\begin{proof}
 For only if part, without loss of generality we will show $M_1$ is a fully $S_1$-idempotent $R_1$-module.
Take a submodule $N_1$ of $M_1$. Then $N_1\times  \{0\}$ is a submodule of $M$. Since $M$ is a fully $S$-idempotent
$R$-module, there exists $s = (s_1, s_2)  \in S_1\times  S_2$ such that $(s_1, s_2)(N_1\times \{0\}) \subseteq (N_1\times  \{0\} :_R M)^2M$. By focusing on first coordinate, we have $s_1N_1 \subseteq (N_1 :_{R_1} M_1)^2M_1$. So $M_1$ is a fully $S_1$-idempotent $R_1$-module. Now assume that $M_1$ is a fully $S_1$-idempotent module and $M_2$ is a fully $S_2$-idempotent module. Take
a submodule $N$ of $M$. Then $N$ must be in the form of $N_1\times  N_2$, where $N_1 \subseteq M_1,N_2 \subseteq M_2$. Since $M_1$ is a fully $S_1$-idempotent $R_1$-module, there exists an $s_1\in S_1$
such that $s_1N_1 \subseteq (N_1 :_{R_1} M_1)^2M_1$. Similarly, there exists an element $s_2 \in S_2$ such that $s_2N_2 \subseteq (N_2 :_{R_2}M_2)^2M_2$. Now, put $s = (s_1, s_2) \in S$. Then we get
$$
(s_1, s_2)N \subseteq s_1N_1\times  s_2N_2 \subseteq (N_1 :_{R_1} M_1)^2M_1\times  (N_2 :_{R_2}
M_2)^2M_2 \subseteq (N :_R M)^2M.
$$
Hence, $M$ is a fully $S$-idempotent $R$-module.
\end{proof}

\begin{thm}\label{t2.13}
Let $M_i$ be an $R_i$-module for $i=1,2,..,n$ and let $S_1, ..., S_n$ be multiplicatively closed subsets of
$R_1, ..., R_n$, respectively. Assume that $M = M_1\times ...\times M_n$, $R = R_1\times ...\times R_n$ and
$S = S_1\times ...\times S_n$. Then the following statements are equivalent.
\begin{itemize}
\item [(a)] $M$ is a fully $S$-idempotent module.
\item [(b)] $M_i$ is a fully $S_i$-idempotent module for each $i\in \{1, 2, ..., n\}$.
\end{itemize}
\end{thm}
\begin{proof}
We use mathematical induction. If $n =1$, the claim is trivial. If $n =2$, the
claim follows from Theorem \ref{l2.12}. Assume that the claim is true for $n<k$ and we will
show that it is also true for $n=k$. Put $M = (M_1\times ... \times M_{n-1})\times  M_n$, $R = (R_1\times  R_2\times ...\times  R_{n-1})\times  R_n$, and $S = (S_1\times ...\times S_{n-1})\times  S_n$. By Theorem \ref{l2.12}, $M$ is fully $S$-idempotent module if and only if $M_1\times ... \times M_{n-1}$ is a fully $(S_1\times ...\times S_{n-1})$-idempotent  $(R_1\times  R_2\times ...\times  R_{n-1})$-module and $M_n$ is a fully $S_n$-idempotent $R_n$-module. Now the rest follows from the induction hypothesis.
\end{proof}

Let $M$ be an $R$-module. The \textit{idealization} or \textit{trivial extension} $R \propto M = R\oplus M$ of $M$ is a commutative
ring with componentwise addition and multiplication $(a,m)(b,\acute{m}) = (ab, a\acute{m}+ bm)$ for each $a, b \in R$, $m,\acute{m} \in M$ \cite{AW09}. If $I$ is an ideal of $R$ and $N$ is a submodule of $M$, then $I \propto N$ is
an ideal of $R \propto M$ if and only if $IM \subseteq N$. In that case, $I \propto N$ is called a \textit{homogeneous ideal} of
$R \propto M$. Also if $S \subseteq R$ is a multiplicatively closed subset, then $S \propto N$ is a multiplicatively closed
subset of $R \propto M$ \cite[Theorem 3.8]{AW09}.

Let $I$ be an ideal of $R$ and $S \subseteq  R$ be a multiplicatively
closed subset of $R$. If $I$ is a fully $S$-idempotent $R$-module, then we say that $I$ is a \textit{fully $S$-idempotent ideal} of $R$.
\begin{thm}\label{t2.14}
Let $N$ be a submodule of an $R$-module $M$ and $S$ be a multiplicatively closed subset of $R$. Then the
following statements are equivalent:
\begin{itemize}
\item [(a)]  $N$ is a fully $S$-idempotent $R$-module;
\item [(b)] $0 \propto N$ is a fully $S \propto 0$-idempotent ideal of $R \propto M$;
\item [(c)] $0 \propto N$ is a fully $S \propto M$-idempotent ideal of $R \propto M$.
\end{itemize}
\end{thm}
\begin{proof}
$(a) \Rightarrow (b)$.
Suppose that $N$ is a fully $S$-idempotent $R$-module. Take an ideal $J$ of $R \propto M$ contained
in $0 \propto N$. Then $J = 0 \propto \acute{N}$ for some submodule $\acute{N}$ of $M$ with $\acute{N} \subseteq N$. Since $N$ is a fully $S$-idempotent module, there exists $s \in S$  with $s\acute{N} \subseteq (\acute{N}:_RN)^2N\subseteq   \acute{N}$. It follows that
$$
(s, 0)J = (s, 0)(0 \propto \acute{N}) = 0 \propto s\acute{N}
\subseteq 0 \propto  (\acute{N}:_RN)^2N =
$$
$$
((\acute{N}:_RN)^2\propto M)(0 \propto N)
\subseteq  0 \propto \acute{N} = J.
$$
This implies that $0 \propto N$ is a fully $S \propto 0$-idempotent ideal of $R \propto M$.

$(b) \Rightarrow (c)$.
This follows from the fact that $S \propto 0 \subseteq  S \propto M$ and Proposition \ref{p2.8} (a).

$(c) \Rightarrow (a)$.
Suppose that $0 \propto N$ is a fully $S \propto M $-idempotent ideal of  $R \propto M $. Let  $\acute{N} $ be a submodule
of  $N $. Then  $0 \propto \acute{N} \subseteq   0 \propto N$ and $0 \propto \acute{N}$ is an ideal of  $R \propto M $. Since  $0 \propto N $ is a fully  $S \propto M $-idempotent ideal of  $R \propto M $, there exists  $(s,m) \in S \propto M $  such that $(s,m)(0 \propto \acute{N}) \subseteq  \acute{J}^2 (0 \propto N)\subseteq   (0 \propto \acute{N} )$, where $\acute{J}=(0 \propto\acute{N}:_{R \propto M}0 \propto N)$. Clearly, $(\acute{N}:_RN) \propto N=(0 \propto\acute{N}:_{R \propto M}0 \propto N)$. Now, set $J = \acute{J} + 0 \propto M$. As
$$
\acute{J} (0 \propto N) = \acute{J} (0 \propto N) + (0 \propto M)(0 \propto N) = (\acute{J} + 0 \propto M)(0 \propto N),
$$
we may assume that $\acute{J}$ contains $0 \propto M$. Then
$\acute{J} = (\acute{N}:_RN) \propto M$. This implies that
$$
(s,m)(0 \propto \acute{N}) = 0 \propto s\acute{N} \subseteq  ((\acute{N}:_RN) \propto M)^2(0 \propto N) =0 \propto (\acute{N}:_RN)^2N \subseteq  0 \propto \acute{N}
$$
and so $s\acute{N} \subseteq  (\acute{N}:_RN)^2N \subseteq  \acute{N}$. Hence, $N$ is a fully $S$-idempotent $R$-module
\end{proof}

\begin{prop}\label{p2.15}
Let $M$ and $\acute{M}$ be $R$-modules. Assume that $S$ is a multiplicatively closed subset of $R$ and
$f : M \rightarrow \acute{M}$ is an $R$-epimorphism. If $M$ is a fully $S$-idempotent module, then $\acute{M}$ is a fully $S$-idempotent
module. Conversely, suppose that $\acute{M}$ is a fully $S$-idempotent module and $tker(f) = 0$ for some
$t \in S$. Then $M$ is a fully $S$-idempotent module.
\end{prop}
\begin{proof}
Let $\acute{N}$ be a submodule of $\acute{M}$. Then $N :=f^{-1}(\acute{N})$ is a submodule of $M$. As $M$ is a fully  $S$-idempotent module, there exists $s \in S$ such that $sN \subseteq  (N : _RM)^2M \subseteq  N$. Hence, $f(sN) \subseteq  f((N :_R
M)^2M) \subseteq  f(N)$. This yields that
$$
s\acute{N} = sf(N) \subseteq  (N : _RM)^2f(M) = (N : _RM)\acute{M} \subseteq  \acute{N}.
$$
Thus $ s\acute{N} \subseteq
(N : _RM)^2\acute{M} \subseteq  \acute{N}$.  Hence, $\acute{M}$ is a fully $S$-idempotent module.
For the converse, let $N$ be a submodule of $M$ and $tker(f) = 0$. Since $f(M)=\acute{M}$ is
a fully $S$-idempotent module, there exists $s\in S$ with $sf(N)\subseteq (f(N):_Rf(M))^2f(M)$. Hence
$$
sN+ker(f)\subseteq (N:_RM)^2M+ker(f).
$$
Multiplying through by $t$ and noting that $tker(f)= 0$
gives that
$$
(st)N\subseteq t(N:_RM)^2M \subseteq (N:_RM)^2M \subseteq N.
$$
So, $M$ is a fully $S$-idempotent $R$-module.
\end{proof}

\begin{cor}\label{c2.15}
Let $S$ be a multiplicatively closed subset of $R$, $M$ a fully $S$-idempotent $R$-module, and $N$ a submodule of $M$. Then $M/N$ is a fully
$S$-idempotent $R$-module. Conversely, if $M/N$ is a fully $S$-idempotent $R$-module and there exists $t \in S$
with $tN = 0$, then $M$ is a fully $S$-idempotent $R$-module.
\end{cor}

\begin{prop}\label{p2.16}
Let $S$ and $T$ be multiplicatively closed subsets of $R$.
Put $\tilde{S} = \{s/1\in T^{-1}R: s \in S\}$, a multiplicatively closed subset of $T^{-1}R$. Suppose that $M$ is a fully $S$-idempotent $R$-module. Then $T^{-1}M$ is a fully $\tilde{S}$-idempotent $T^{-1}R$-module. Hence if $S \subseteq  T^*$, then $T^{-1}M$ is a fully idempotent $T^{-1}R$-module. Thus $S^{-1}M$ is a fully idempotent $S^{-1}R$-module.
\end{prop}
\begin{proof}
Let $N$ be a $T^{-1}R$-submodule of $T^{-1}M$, so $N = T^{-1}\acute{N}$ for some submodule $\acute{N}$ of $M$. Since
$M$ is a fully $S$-idempotent module, there exists $s \in S$ with $s\acute{N} \subseteq  (\acute{N}:_RM)^2M \subseteq  \acute{N}$. Then
$$
(s/1)N = T^{-1}(s\acute{N}) \subseteq  (T^{-1}(\acute{N}:_RM)^2)(T^{-1}M) \subseteq  T^{-1}\acute{N} = N.
$$
So $T^{-1}M$ is a fully $\tilde{S}$-idempotent $T^{-1}R$-module. If
$S \subseteq  T^*$, then $\tilde{S} \subseteq  U(T^{-1}R)$. Hence, $T^{-1}M$ is a fully idempotent $T^{-1}R$-module by Proposition \ref{p2.3}.
\end{proof}

\begin{cor}\label{c2.17}
Let $M$ be an $R$-module and $S$ a multiplicatively closed subset of $R$ satisfying the maximal
multiple condition (e.g., $S$ is finite or $S \subseteq  U(R)$). Then $M$ is a fully $S$-idempotent module if and
only if $S^{-1}M$ is a fully idempotent $S^{-1}R$-module.
\end{cor}
\begin{proof}
$(\Rightarrow)$:
 This follows from Proposition \ref{p2.16}.

$(\Leftarrow)$:
 Assume that $S^{-1}M$ is a fully idempotent $S^{-1}R$-module. Take a submodule $N$ of $M$. Since
$S^{-1}M$ is a fully idempotent $S^{-1}R$-module, we have
$$
S^{-1}N = (S^{-1}N:_{S^{-1}R}S^{-1}M)^2(S^{-1}M) = S^{-1}((N:_RM)^2M).
$$
Choose $s \in S$ with $t\mid s$ for each $t \in S$. Note that for each $m \in N$, we have $m/1 \in S^{-1}N = S^{-1}((N:_RM)^2M)$ and so there exists $t \in S$ such that $tm \in (N:_RM)M$ and hence $sm \in (N:_RM)M$. Therefore, we obtain
$$
s^2 N \subseteq  s(N:_RM)^2M\subseteq  (N:_RM)^2M \subseteq  N.
$$ Hence, $M$ is a fully $S$-idempotent module.
\end{proof}

Let $S$ be a multiplicatively closed subset of $R$ and $M$ be an $R$-module. A submodule $N$ of $M$ is said to be \emph {pure} if $IN=N \cap IM$ for every ideal $I$ of $R$ \cite{AF74}.
$M$ is said to be \emph{fully pure} if every submodule of $M$ is pure \cite{AF122}.
A submodule $N$ of $M$  is said to be \emph {$S$-pure} if there exists an $s \in S$ such that $s(N \cap IM) \subseteq IN$ for every ideal $I$ of $R$ \cite{FF2022}. $M$ is said to be \emph{fully $S$-pure} if every submodule of $M$ is $S$-pure \cite{FF2022}.
\begin{thm} \label{t2.18}
Let $S$ be a multiplicatively closed subset of $R$,  $M$ an $S$-multiplication $R$-module, and $N$ be a submodule of $M$.
Then the following statements are equivalent.
\begin {itemize}
\item [(a)] $N$ is an $S$-pure submodule of $M$.
\item [(b)] $N$ is an $S$-multiplication $R$-module and
$N$ is an $S$-idempotent submodule of $M$.
\item [(c)] $N$ is an $S$-multiplication $R$-module and there exists an $s \in S$ such that
$sK\subseteq (N:_RM)K$, where $K$ is a submodule of $N$.
\item [(d)] $N$ is an $S$-multiplication $R$-module and there exists an $s \in S$ such that
$s(K:_RN)N\subseteq (K:_RM)(N:_RM)M$, where $K$ is a submodule of $M$.
\end {itemize}
\end{thm}
\begin{proof}
$(a) \Rightarrow (b)$.
 Let $K$ be a submodule of $N$. As $M$ is an $S$-multiplication module, there exists an $s \in S$ such that $sK \subseteq (K:_RM)M$.
Now since $N$ is $S$-pure, there is an $t \in S$ such that $(K:_RN)N\supseteq t(N \cap (K:_RN)M)$. Hence,
$$
(K:_RN)N\supseteq t(N \cap (K:_RN)M)\supseteq t(N \cap (K:_RM)M)
$$
$$
\supseteq t(N \cap sK)=tsK.
$$
This implies that $N$ is an $S$-multiplication $R$-module. Since $M$ is an $S$-multiplication module, there exists an $u \in S$ such that $uN \subseteq (N:_RM)M$. Now as $N$ is $S$-pure, there is an $v \in S$ such that $(N:_RM)uN \supseteq v(N \cap u(N:_RM)M)$. Therefore,
$$
(N:_RM)^2M=(N:_RM)(N:_RM)M\supseteq (N:_RM)uN
$$
$$
\supseteq v(N \cap u(N:_RM)M)=vu(N:_RM)M\supseteq vu^2N.
$$
So, $N$ is an $S$-idempotent submodule.

$(b) \Rightarrow (c)$.
Let $K$ be a submodule of $N$.  Since $N$ is an $S$-multiplication $R$-module, there exists an $s \in S$ such that $sK \subseteq (K:_RN)N$. As $N$ is $S$-idempotent, there is an $t\in S$ such that $tN \subseteq (N:_RM)^2M$. Therefore,
$$
tsK \subseteq t(K:_RN)N= (K:_RN)tN
$$
$$
\subseteq (K:_RN) (N:_RM)^2M=
(N:_RM)(K:_RN) (N:_RM)M
$$
$$
\subseteq (N:_RM)(K:_RN) N\subseteq (N:_RM)K.
$$

$(c) \Rightarrow (a)$. Let $I$ be an ideal of $R$. Since $IM \cap N
\subseteq N$, by part (c), there is an $s \in S$ such that $s(IM \cap N)\subseteq (N:_RM)(IM \cap N)$. Hence,
$$
(N \cap IM)(N:_RM)
\subseteq (N:_RM)IM=IN
$$
implies that $s(IM \cap N)\subseteq IN$. Thus, $N$ is an $S$-copure submodule of $M$.

$(b) \Rightarrow (d)$. Let $K$ be a submodule of $M$. Since $N$
is $S$-idempotent,  there is an $s \in S$ such that $sN\subseteq  (N:_RM)^2M$
$$
s(K:_RN)N\subseteq  (K:_RN)(N:_RM)^2M\subseteq (K:_RM)(N:_RM)M.
$$

$(d) \Rightarrow (b)$. Take $K=N$.
\end{proof}

\begin{cor}\label{c2.20}
Let $S$ be a multiplicatively closed subset of $R$ and $M$ be an $R$-module. Then we have the following.
\begin{itemize}
\item [(a)] If $M$ is a fully $S$-idempotent $R$-module,
  then $M$ is a fully $S$-pure $R$-module.
\item [(b)] If $M$ is an $S$-multiplication fully $S$-pure $R$-module,
  then $M$ is a fully $S$-idempotent $R$-module.
\end{itemize}
\end{cor}
\begin{proof}
(a) By Proposition \ref{p2.8},  every submodule of
$M$ is a fully $S$-idempotent $R$-module. Hence, by Lemma \ref{l2.5}, every submodule of
$M$ is an $S$-multiplication $R$-module. Now the result follows from Theorem \ref{t2.18} $(b) \Rightarrow (a)$.

(b) This follows from Theorem \ref{t2.18} $(a) \Rightarrow (b)$.
\end{proof}

Let $S$ be a multiplicatively closed subset of $R$ and $M$ be an $R$-module.
A submodule $N$ of $M$ is said to be \emph{copure} if $(N:_MI)=N+(0:_MI)$ for every ideal $I$ of $R$ \cite{AF09}.
A submodule $N$ of $M$ is said to be \emph {$S$-copure} if there exists an $s \in S$ such that $s(N:_MI)\subseteq N+(0:_MI)$ for every ideal $I$ of $R$ \cite{FF2023}. $M$  is said to be \emph{fully $S$-copure} if every submodule of $M$ is $S$-copure \cite{FF2023}.

\begin{prop} \label{p2.19}
Let $S$ a multiplicatively closed subset of $R$ and $M$ be an $S$-multiplication $R$-module. If
$N$ is an $S$-copure submodule of $M$, then $N$ is $S$-idempotent.
\end{prop}
\begin{proof}
Suppose that $N$ is an
$S$-copure submodule of $M$. Then
 there exists an $s \in S$ such that
$$
s(N:_M(N:_RM))\subseteq N+(0:_M(N:_RM)).
$$
This in turn implies that $sM\subseteq N+(0:_M(N:_RM))$. It follows that
$$
s(N:_RM)M\subseteq (N:_RM)N.
$$
As $M$ is an $S$-multiplication module,
there is an $t \in S$ such that $tN \subseteq (N:_RM)M$. Hence,
we have
$$
st^2N \subseteq st(N:_RM)M\subseteq (N:_RM)tN\subseteq (N:_RM)^2M,
$$
as needed.
\end{proof}

\textbf{Acknowledgments.} The author would like to thank Prof. Habibollah Ansari-Toroghy for his helpful suggestions and useful comments.

\end{document}